\documentclass[11pt,a4paper]{amsart}
\usepackage{epsfig,a4wide}

\newtheorem{theorem}{Theorem}

\begin{document}

\title{An upper bound on stick numbers of knots}

\author[Y. Huh]{Youngsik Huh}
\address{Department of Mathematics, School of Natural Sciences,
Hanyang University, Seoul 133-791, Korea}
\email{yshuh@hanyang.ac.kr}
\author[S. Oh]{Seungsang Oh}
\address{Department of Mathematics, Korea University,
Anam-dong, Sungbuk-ku, Seoul 136-701, Korea}
\email{seungsang@korea.ac.kr}
\thanks{This work was supported by the Korea Research Foundation Grant
funded by the Korean Government (MOEHRD, Basic Research Promotion
Fund) (KRF-2006-312-C00469).}

\keywords{knot, stick number, upper bound} \subjclass{57M25 57M27}

\begin{abstract}
In 1991, Negami found an upper bound on the stick number $s(K)$ of
a nontrivial knot $K$ in terms of the minimal crossing number
$c(K)$ of the knot which is $s(K) \leq 2 c(K)$. In this paper we
improve this upper bound to $s(K) \leq \frac{3}{2} (c(K)+1)$.
Moreover if $K$ is a non-alternating prime knot, then $s(K) \leq
\frac{3}{2} c(K)$.
\end{abstract}

\maketitle

\section{Introduction} \label{sec:intro}
A simple closed curve embedded into the Euclidean 3-space is
called a {\em knot}. Two knots $K$ and $K^{\prime}$ are said to be
{\em equivalent}, if there exists an orientation preserving
homeomorphism of $\mathbb{R}^3$ which maps $K$ to $K^{\prime}$, or
to say roughly, we can obtain $K^{\prime}$ from $K$ by a sequence
of moves without intersecting any strand of the knot. And the
equivalence class of $K$ is called the {\em knot type of $K$}. A
knot equivalent to another knot in a plane of the 3-space is said
to be {\em trivial}. A {\em stick knot} is a knot which consists
of finite line segments, called {\em sticks}.

One natural question concerning stick knots may be the {\em stick
number} $s(K)$ of a knot $K$ which is defined to be the minimal
number of sticks for construction of the knot type into a stick
knot. Since this representation of knots has been considered to be
a useful mathematical model of cyclic molecules or molecular
chains, the stick number may be an interesting quantity not only
in knot theory of mathematics, but also in chemistry and physics.
Although it seems to be not easy to determine $s(K)$ completely
for arbitrary knot $K$, which is usual for any other minimality
invariats of knots, there are some literatures in which the range
of $s(K)$ was theoretically investigated \cite{ABGW, FLS, J, N,
R}. Especially, in 1991, Negami found upper and lower bounds on
the stick number of any nontrivial knot $K$ in terms of the
crossing number $c(K)$ \cite{N}:
\[ \frac{5+ \sqrt{25+8(c(K)-2)}}{2} \leq s(K) \leq 2c(K) \]
\noindent Here the crossing number $c(K)$ is  the minimal number
of double points in any generic projection of the knot type into
the plane $\mathbb{R}^2 \subset \mathbb{R}^3$. In \cite{A, FLS} it
was questioned whether it is possible to improve the Negami's
inequalities: To describe specifically,
\begin{enumerate}
\item[Q1.] Is there any knot satisfying $2s(K)=
5+\sqrt{25+8(c(K)-2)}$ ?
\item[Q2.] Is there any knot satisfying $s(K)= 2c(K)$ other than the trefoil knots?
\end{enumerate}
The Negami's lower bound was slightly improved by Calvo \cite{CA}.
And recently Elifai showed that the answer for the first question
is negative for the knots with $c(K)\leq 26$ \cite{E}.

In this paper we give an improved upper bound on $s(K)$ and answer
for the second question. The following theorem is the main result
of this paper.
\begin{theorem} \label{thm:main}
Let $K$ be any nontrivial knot. Then $s(K) \leq \frac{3}{2}
(c(K)+1)$. Moreover if $K$ is a non-alternating prime knot, then
$s(K) \leq \frac{3}{2} c(K)$.
\end{theorem}
Note that $c(K)$ is at least three for any nontrivial knot $K$. If
$\frac{3}{2} (c(K)+1) \geq 2c(K)$, then $c(K)\leq3$. It is known
that the {\em right-handed trefoil knot} in Figure \ref{fig1} and
its mirror image(called {\em left-handed trefoil}) are the only
knot types with $c(K)=3$. And their stick numbers are exactly six.
Therefore, from our theorem, we can conclude that the second
question is negative for any knot type other than the trefoil
knots.

In the rest of this paper we will prove Theorem \ref{thm:main} by
investigating the relation among stick number $s(K)$, crossing
number $c(K)$ and another minimality quantity, {\em arc index}
$a(K)$.
\begin{figure}
\epsfbox{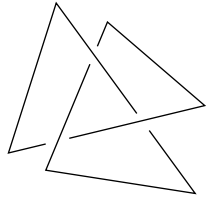} \caption{A projection, a diagram and a stick
representation of the right-handed trefoil} \label{fig1}
\end{figure}
\section{Proof of main theorem}
We introduce some definitions necessary for the proof of Theorem
\ref{thm:main} . A continuous map from the unit circle into
$\mathbb{R}^2$ is called a {\em knot projection}, if each multiple
point of the map is a transversal double point which will be
called a {\em crossing}. By adding the information on the relative
height of two strands at each crossing into the projection, we
obtain a {\em diagram} representing a knot. An example of a
diagram is given in Figure \ref{fig1}. The {\em crossing number
$c(K)$} of a knot $K$ is defined to be the minimal number of
crossings among all diagrams representing the knot type.

In our proof we consider a specific type of knot diagram which is
obtained by drawing $n$ chords $l_1, \ldots, l_n$ on a
2-dimensional disk $B$ according to the following rules:
\begin{enumerate}
\item The end points of each $l_i$ lie on the boundary of $B$.
\item If $l_i$ and $l_j$ share a crossing in the interior of $B$ and $i<j$, then $l_i$
underpasses $l_j$.
\end{enumerate}
\noindent If a diagram of such type represents a knot $K$, it is
called an {\em arc presentation} of $K$. And the {\em arc index
$a(K)$} of a knot $K$ is defined to be the minimal number of
chords among all possible arc presentations of its knot type. In
fact our definition of arc presentation is a little modified from
the original one, but essentially identical \cite{BM,CR}. The left
of Figure \ref{fig2} shows an arc presentation of the trefoil
knot.

Bae and Park established an upper bound of arc index in terms of
the crossing number. Corollary $4$ and Theorem $9$ in \cite{BP}
provide the following;
\begin{theorem}[Bae and Park] \label{thm:BP}
Let $K$ be any nontrivial knot. Then $a(K) \leq c(K)+2$. Moreover
if $K$ is a non-alternating prime knot, then $a(K) \leq c(K)+1$.
\end{theorem}

Therefore, if we prove Theorem \ref{thm:2}, then the proof of our
main theorem is completed.
\begin{theorem} \label{thm:2}
Let $K$ be any nontrivial knot. Then $s(K) \leq \frac{3}{2}
(a(K)-1)$.
\end{theorem}
\begin{proof}
Let $K$ be a nontrivial knot with $a(K)=n$ and $D$ be an its arc
presentation with $n$ chords $l_1, \ldots, l_n$. $\overline{D}$
denotes the projection of $D$. Let $\pi:\mathbb{R}^3 \rightarrow
\mathbb{R}^2$ be the map defined by $\pi(x,y,z)=(x,y)$. From $D$
we construct a stick knot $K_1$ in the cylinder $B \times [1,n]$
so that $\pi(K_1)=\overline{D}$. For each integer $i\in [1,n]$,
put a line segment $h_i$ into $B\times\{i\}$ so that
$\pi(h_i)=l_i$. If $l_i \cap l_j \cap
\partial B = \{p\}$, then connect $\pi^{-1}(p)\cap h_i$ to
$\pi^{-1}(p)\cap h_j$ by a vertical line segment $v_{ij}$ so that
$\pi(v_{ij})=p$. Note that we do not distinguish $v_{ij}$ from
$v_{ji}$. By adding all such vertical sticks, we obtain a stick
knot $K_1$ with $2n$ sticks which is equivalent to $K$. Figure
\ref{fig2} shows an example of a stick knot constructed from an
arc presentation  of the right-handed trefoil.

A horizontal stick $h_i$ is said be {\em type-I} ({\em resp.
type-III}), if the indices of the two chords adjacent to $l_i$ in
$D$ are greater ({\em resp.} less) than $i$. If neither type-I nor
type-III, then $h_i$ is {\em type-II}. Notice that $h_1$ and $h_n$
should be type-I and type-III, respectively. If $h_{n-1}$ is
type-II, then we can modify $K_1$ as illustrated in Figure
\ref{fig3}(a), so that the number of horizontal sticks is reduced
by one, which is contradictory to the minimality of the number of
chords. Since $K_1$ is a nontrivial knot, $h_{n-1}$ can not be
type-I. Hence $h_{n-1}$ should be type-III and similarly $h_2$
should be type-I.

\begin{figure}
\epsfbox{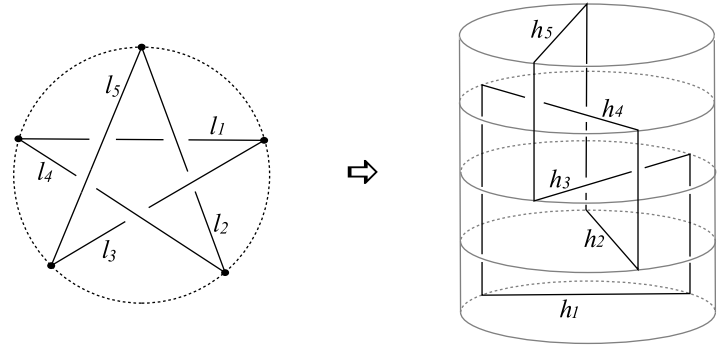} \caption{A stick knot in cylinder constructed
from an arc presentation} \label{fig2}
\end{figure}

From $K_1$ we construct another stick knot $K_2$ in which the
$z$-coordinate of each $h_i$ may be changed into some integer
$z_i$, while its $xy$-coordinates are preserved. Concretely, if we
denote the $i$-th horizontal stick of $K_2$ by $h^{\prime}_i$,
then $\pi(h_i)=\pi(h^{\prime}_i)$ and $h^{\prime}_i \subset
B\times \{z_i\}$ in $B\times [0,\infty)$. The height $z_i$ will be
determined in inductive way. Firstly, set $z_1=1$ and $z_2=2$. For
$3 \leq i \leq n$, if $h_i$ is type-I, then $z_i$ is set to be
$z_{i-1}+1$. If $h_i$ is type-II, there is a vertical stick
$v_{ij}$ with with $j<i$ which is adjacent to $h_i$. Then put
$h^{\prime}_i$ into $B\times \{z_i\}$ for some large enough $z_i$
and connect $h^{\prime}_i$ to $h^{\prime}_j$ via the vertical
stick $v^{\prime}_{ij}$ between $B\times \{z_i\}$ and $B\times
\{z_j\}$, so that the interior of the triangle determined by
$h^{\prime}_i \cup v^{\prime}_{ij}$ has no intersection with any
other horizontal stick $h^{\prime}_k$, $k < i$. Here, such a
triangle will be called a {\em reducible triangle} of
$h^{\prime}_i$. If $h_i$ is type-III, that is, $h_i$ is adjacent
to some $v_{ij}$ and $v_{ik}$ with $i>j>k$, then similarly the
height of $h_i^{\prime}$ is determined so that the triangle whose
boundary contains $h^{\prime}_i \cup v^{\prime}_{ij}$ is
reducible.

Now we modify $K_2$ in purpose to decrease the number of sticks.
For each $i$ from $3$ to $n-1$, if $h^{\prime}_i$ is type-II or
III, replace $h^{\prime}_i \cup v^{\prime}_{ij}$ with the other
edge of the reducible triangle (See Figure \ref{fig3}(b)). Since
the interior of the triangle has no intersection with any other
part of the knot, such replacement preserves the knot type. And
the number of sticks becomes reduced by one, after each
modification. For $h^{\prime}_n$, we modify the knot in another
way. Let $v^{\prime}_{ni}$ and $v^{\prime}_{nj}$ be the sticks
adjacent to $h^{\prime}_n$. The other stick adjacent to
$v^{\prime}_{ni}$ ({\em resp.} $v^{\prime}_{nj}$) is denoted by
$e_i$ ({\em resp.} $e_j$). Extend $e_i$ and $e_j$ toward the end
points $e_i\cap v^{\prime}_{ni}$ and $e_j\cap v^{\prime}_{nj}$,
respectively, long enough so that the two extended line segments
are connected by a line segment outside of $B\times[1,z_n]$.
Replace $e_i \cup v^{\prime}_{ni} \cup h^{\prime}_{n} \cup
v^{\prime}_{nj} \cup e_j$ with these three line segments(See
Figure \ref{fig3}(c) for example). Then the knot type is
preserved, but the number of sticks is reduced by two. Let $K_3$
be the resulting stick knot.
\begin{figure}
\epsfbox{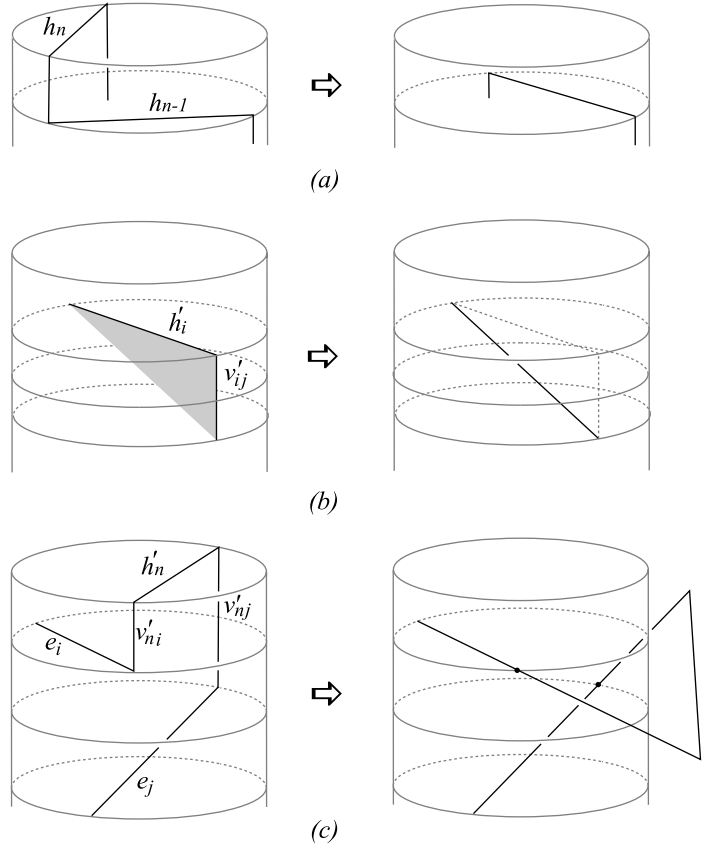} \caption{(a) Reduction when $h_{n-1}$ is
type-II, (b) Reduction along a reducible triangle and (c)
Reduction near $h^{\prime}_n$} \label{fig3}
\end{figure}

Let $\beta_1(K_1)$, $\beta_2(K_1)$ and $\beta_3(K_1)$ be the
numbers of type-I, type-II and type-III horizontal sticks of
$K_1$, respectively. Note that $\beta_1(K_1)=\beta_3(K_1)$. Since
$n=\beta_1(K_1)+\beta_2(K_1)+\beta_3(K_1)$, the number of sticks
of $K_3$ is equal to
$$2n-\beta_2(K_1)-(\beta_3(K_1) -1) -2 = n+ \beta_1(K_1) -1 \;\;\;.$$
Therefore,
$$s(K) \leq n+\beta_1(K_1) -1 \;\;\;.$$
Now we consider an upper bound of $\beta_1(K_1)$. If $n$ is odd,
then $\beta_1(K_1) \leq (n-1)/2$. If $n$ is even, then
$\beta_1(K_1) \leq n/2$ in which the equality holds only when
$\beta_2(K_1)=0$. In that case, let $v_{i1}$ and $v_{j1}$ be the
horizontal sticks adjacent to $h_1$ in $K_1$. And replace $v_{i1}
\cup h_1 \cup v_{j1}$ with $v_{i(n+1)}\cup h_{n+1} \cup
v_{j(n+1)}$, where $h_{n+1}$ is the horizontal line segment in
$B\times \{n+1\}$ satisfying $\pi(h_1)=\pi(h_{n+1})$. Then the
resulting stick knot $K_1^{\prime}$ in $B\times[2,n+1]$ is
equivalent to $K_1$. Because $\beta_2(K_1)=0$, we have
$$\beta_1(K_1^{\prime}) = \beta_1(K_1)-1 \leq
\frac{n}{2}-1 < \frac{n-1}{2} \;\;\;.$$

To summarize, for a nontrivial knot $K$ with $a(K)=n$, there
exists an equivalent stick knot $K^{\prime}$ with 2$n$ sticks in
the cylinder satisfying
$$\beta_1(K^{\prime}) \leq \frac{n-1}{2}$$
and therefore
$$s(K) \leq n+\beta_1(K^{\prime})-1 \leq a(K)+\frac{a(K)-1}{2}-1 =
\frac{3}{2}(a(K)-1) \;\;\;.$$
\end{proof}
\noindent{\bf Acknowledgement.} The second author(corresponding
author) would like to give thanks to Professor Sang Jin Lee for
his administrative support to the research project related with
this paper.

\end{document}